\documentclass[11pt]{article}

\usepackage{latexsym,amsfonts,amssymb,amsmath,amsthm}
\usepackage{graphicx}
\usepackage{color}
\usepackage{mathrsfs}
\usepackage{amsmath, amsthm, amsfonts, amssymb, cite}
\usepackage{wrapfig}
\usepackage{stackrel}
\usepackage{color}
\usepackage{float}
\usepackage{enumitem}
\usepackage{mathtools}
\usepackage{multicol}
\usepackage[normalem]{ulem}
\usepackage{amsmath,amsthm,verbatim,amssymb,amsfonts,amscd,graphicx}
\usepackage{graphics}
\usepackage{relsize}
\usepackage{esvect}
\usepackage{mathtools}
\usepackage{physics}
\usepackage{multirow}

    \newtheoremstyle{TheoremNum}
        {\topsep}{\topsep}              
        {\itshape}                      
        {}                              
        {\bfseries}                     
        {.}                             
        { }                             
        {\thmname{#1}\thmnote{ \bfseries #3}}
    \theoremstyle{TheoremNum}

\theoremstyle{plain}
\newtheorem{theorem}{Theorem}
\newtheorem*{theorem*}{Theorem}
\newtheorem{corollary}{Corollary}
\newtheorem{lemma}{Lemma}
\newtheorem{proposition}{Proposition}

\newtheorem{question}{Question} 
\theoremstyle{definition}

\newtheorem{remark}{Remark}

\newtheorem*{ack}{Acknowledgements}
\begin{document}
\title{On the regularized $L^4$ norm for Eisenstein series in the level aspect}
\author{Jiakun Pan}
\maketitle

\section{Introduction}
One important goal of  modern number theory is to understand the familial limiting behavior of automorphic forms. When the behavior under observation is about the mass distribution of these family members, the Random Wave Conjecture (\textbf{RWC}) of Hejhal and Rackner \cite{HR} is a general heuristic. For an instant illustration, let us fix a compact subset $\Omega$ of $SL_2(\mathbb{Z})\backslash\mathbb{H}$, take $d\mu$ to be the hyperbolic probablistic measure, and write $\{ u_j \}_{j\geq 1}$ as the sequence of Hecke-Maass cuspforms of increasing order of eigenvalues $\lambda_j$, with renormalization $\int_{\Omega} |u_j|^2 d\mu =1$. Then according to RWC, as $j\rightarrow \infty$, we have for any even $p\geq 0$ that
\begin{align}\label{RWC}
\int_{\Omega} |u_j|^p d\mu \sim c_p := \frac{1}{\sqrt{2\pi}} \int_{-\infty}^{\infty} x^{p}e^{-\frac{x^2}{2}}dx,
\end{align}
also known as the \textit{Gaussian Moments Conjecture}. It is obvious that the cases $p=0,2$ are trivial. For $p=4$, Buttcane and Khan \cite{BuK} proved \eqref{RWC}, assuming the Generalized Lindel\"of Hypothesis, while Humphries and Khan \cite{HK} worked it out for dihedral forms unconditionally.

RWC for the second moment is closely connected to the Quantum Unique Ergodicity (\textbf{QUE}) conjecture of Rudnick and Sarnak \cite{RS}, which concerns mass distribution on negatively curved manifolds in general. Notably, the first unconditionally obtained QUE result is not for $u_j$, but $E_t:=E(z,\frac{1}{2}+it)$, the real analytic Eisenstein series. In terms of RWC for $p=2$, Luo and Sarnak \cite{LS} showed that if we replace $u_j$ with $\widetilde{E_t}:= (\tfrac{3}{\pi} \log(t^2+1))^{-1/2} E_t$, then \eqref{RWC} is still true for all $\Omega$ compact and Jordan measurable, as $t\rightarrow \infty$. When it comes to the case $p=4$, RWC says the fourth moment of $\widetilde{E}_t$ on $\Omega$ is asymptotically $c_4=3$,
and the best result is an upper bound $O_{\Omega}(\log^2 t)$, due to Spinu \cite{Sp} and Humphries \cite{H}.

Besides truncating the domain, there are alternative ways to deal with the divergence of $|E_t|^2$. Another treatment worthy of consideration is regularized integral invented by Zagier [Z] and developed by Michel and Vankatesh [MV], which provides nice period integrals to proceed. In this regard, Djankovi\'c and Khan \cite{DK1}, \cite{DK2} showed as $t\rightarrow \infty$,
\begin{align*}
\langle |E_t|^2, |E_t|^2 \rangle_{\mathrm{reg}} = \frac{72}{\pi} \log^2 t + O(\log^{\frac{3}{5}}t),
\end{align*}
for $\langle f, g \rangle_{\mathrm{reg}}:= \int_X^{\mathrm{reg}} f(z)\overline{g}(z) y^{-2} dxdy$ (see Section \ref{reduction} for definition).
\begin{remark}
One may as well consider the truncated Eisenstein series $E^Y$ obtained by some certain treatments on $E$. In fact, the $L^4$-norm of $E^Y$ is so closely related to the regularized integral, that they have also proved RWC for $E^Y$ and $p=4$. See \cite[thm 1.2]{DK1} for the detailed arguments.
\end{remark}

As is mentioned, RWC with $p=2$ is related to QUE, to which there are level aspect variations. For $\Gamma_0(N)$ Eisenstein series $E$ of primitive nebentypus $\chi$ mod $N$, the author and Young \cite{PY} obtained
\begin{align*}
\frac{1}{\mathrm{Vol }(\Omega)} \int_{\Omega} |E|^2 y^{-2}dxdy = \frac{6}{\pi} \log N + \frac{12}{\pi} \Re\frac{L'}{L}(s,\chi) + O(N^{-\frac{1}{8}+\varepsilon}),
\end{align*}
for some $s$ on the $1$-line. Note this result is unconditional, and under GRH the second term is $O(\log\log N)$. Since Eisenstein series on high levels always concerns Dirichlet $L$-functions instead of Riemann Zeta functions in the $t$-aspect, it would be not realistic to expect an asymptotic formula for the truncated $L^4$-norm of these Eisenstein series. For other automorphic forms, the level aspect QUE for holomorphic forms is proved by Nelson \cite{N1} and Nelson-Pitale-Saha \cite{NPS}, and Nelson \cite{N2} proved QUE on prime level Maass forms, up to some reasonable hypotheses. 

Since it is obvious from the definition of regularized integrals, that $\langle E, E \rangle_{\mathrm{reg}} =0$ for all $\Gamma_0(N)$ Eisenstein series, the non-trivial cases for the regularized $L^4$-norm begin with $p=4$. This paper is to partly address this problem.

\begin{theorem}
Let $N>1$, and $E= E_{\mathfrak{a}}(z,s,\chi)$ be a $\Gamma_0(N)$-Eisenstein series attached to an Atkin-Lehner cusp $\mathfrak{a}$, and of nebentypus $\chi$ primitive mod $N$.  As $N\rightarrow\infty$, we have
\begin{align*}
    \langle |E_{\mathfrak{a}}(\cdot,\frac{1}{2}+iT,\chi)|^2, |E_{\mathfrak{a}}(\cdot,\frac{1}{2}+iT,\chi)|^2 \rangle_{\mathrm{reg}} = I_1 + I_2,
\end{align*}
where (writing $\mathcal{O}_j(N)$ for an orthonormal basis for the space of level $N$ cusp forms of eigenvalue $\lambda_j = \tfrac{1}{4}+t_j^2$, and $\Lambda(s,f)=(\tfrac{q(f)}{\pi})^{\frac{s}{2}}\Gamma(\tfrac{s+\kappa_1(f)}{2})\Gamma(\tfrac{s+\kappa_2(f)}{2}) L(s,f)$ for the completed $L$-function)
\begin{multline}\label{I_1}
    \nu(N) I_{1}= \sum_{j\geq 1} \frac{\cosh(\pi t_j)}{2} \sum_{u_j \in \mathcal{O}_j(N)}  B(u_j,T) \frac{\Lambda^2(\frac{1}{2},u_j) |\Lambda(\frac{1}{2}+2iT, u_j\otimes \psi)|^2}{L(1,\mathrm{sym}^2 u_j) |\Lambda(1+2iT,\psi)|^4} \\
+ \frac{\nu(N)}{4\pi} \sum_{\mathfrak{b}} \int_{-\infty}^{\infty} \Big| \langle |E|^2, E_{\mathfrak{b}}(\cdot, \frac{1}{2}+it) \rangle_{\mathrm{reg}} \Big|^2 dt
\end{multline}
for some primitive $\psi = \psi(\chi,\mathfrak{a})$ mod $N$ to be defined in the next section, $B(u_j, T)\ll N^{\varepsilon}$ for some neglegible values from the bad Euler factors over primes dividing level of $u_j$, and 
\begin{align}\label{I_2}
    \nu(N) I_2 = \frac{24}{\pi}\log^2 N + O(\frac{L''}{L}(1+2iT,\psi) + (\log N + \frac{L'}{L}(1+2iT, \psi)) \log\log N).
\end{align}
\end{theorem}
\begin{remark}
The multiplication by $\nu(N)$ to $I_1$ and $I_2$ is under consideration of $L^4$-renormalization. That is to say, if we regard $E_{\mathfrak{a}}$ to be ``$L^2$-normalized" (they do have comparable behaviors with the classical Eisenstein series $E(z,\frac{1}{2}+it)$ in the $t$-aspect, see \cite{LS} and \cite{PY} for a QUE comparison), then we should expect $\int^{\mathrm{reg}}_{\scriptscriptstyle{\Gamma_0(N)\backslash\mathbb{H}}} |E_{\mathfrak{a}}|^4$ to have size $\asymp \nu(N)^{-1}$.
\end{remark}

To take a closer look at the above theorem, we firstly point it out that $\nu(N) I_2 = O(N^{\varepsilon})$ unconditionally, due to Siegel. For $\nu(N) I_1$, we are in short of methods to attack the continuous contribution in general. When $N$ is square-free, we can adopt an alternative ``formal orthonormal basis" for the Eisenstein series (see \cite[sec 5]{PY}), each of which works well with Fricke involution, making the period integrals easily calculable, and in fact, they are relatively small, just like in \cite[lma 4.4]{DK1}. Therefore, to derive an upper bound for the regularized fourth moment of $E$, it boils down to estimate the discrete contribution of $\nu(N)I_1$.

Going one step further, we are able to see (assuming $N$ is square-free)
\begin{align}\label{centraltask}
\nu(N)I_1 \sim \sum_{t_j \leq 2T + O(\log^2 T)} \sum_{u_j\in \mathcal{O}_j(N)} L^2(\tfrac{1}{2},u_j) |L(\tfrac{1}{2}+2iT, u_j \otimes \psi)|^2.
\end{align}
We think it is not technologically possible to obtain an asymptotic formula for it now, as the ``log ratio" here is $6$, while that number in \cite{DK2} is $4$. For this sum, the spectral large sieve implies
\begin{align*}
\sum_{t_j \leq 2T+O(\log^2 T)} \sum_{u_j} |L(\tfrac{1}{2}+2iT,u_j\otimes \psi)|^2 \ll_{_T,_\varepsilon} N^{1+\varepsilon},
\end{align*}
any subconvexity bound for $\max_{u_j} L^2(\tfrac{1}{2},u_j)$ will translate to an upper bound on $\nu(N)I_1$. The current best bound for $L(\tfrac{1}{2},u_j)$ is $O(N^{\frac{5}{24}+\frac{\theta}{12}})$ due to Blomer and Khan\footnote{They require $N$ to be relatively prime with $6$.} \cite{BlK}, where $\theta = \tfrac{7}{64}$ is the best known exponent such that $|\lambda_j(p)| \ll p^{\theta}+p^{-\theta}$ uniformly, by Kim and Sarnak \cite{KS}.
\begin{corollary}
Unconditionally\footnote{Needless to say, if we assume GRH, then the $O(N^{\varepsilon})$ bound is for granted.}, we have for large square-free $N$ 
\begin{align*}
\langle |E|^2, |E|^2 \rangle_{\mathrm{reg}} = O(N^{\frac{5}{12}+\frac{\theta}{6}+\varepsilon}).
\end{align*}
\end{corollary}

\begin{question}
Can we improve the estimation for \eqref{centraltask}?
\end{question}
In a following paper, we want to provide a sharper bound for $N$ prime.

\begin{remark}\label{GRH}
Assuming GRH, we can see $\nu(N) I_2 \sim \frac{24}{\pi} \log^2 N$, which is in agreement with \cite{DK1} in the spectral aspect. This makes us wonder if $\nu(N)I_1 \sim \tfrac{48}{\pi} \log^2 N$ also holds, corresponding to \cite{DK2}, and validating a level aspect analogue of RWC for $p=4$.
\end{remark}

\begin{ack}
The author thanks Roman Holowinsky and Rizwanur Khan for initial encouragements, and Peter Humphries, Sheng-Chi Liu and Ian Petrow for useful comments. He is supported, both academically and financially, by Dr. Matthew P. Young, his PhD advisor.
\end{ack}

\section{General strategy and spectral decomposition} \label{reduction}
Just like what happened in \cite{DK1}, \cite{Y} and \cite{PY}, $|E_{\mathfrak{a}}(z,\frac{1}{2}+iT,\chi)|^2$ is not directly regularizable, because we cannot subtract it by $E_{\mathfrak{a}}(z,1)$, which is not defined. Instead, we need to consider rewriting the $L^4$-norm as $\langle E_{\mathfrak{a}}(\cdot, s_1,\chi)E_{\mathfrak{a}}(\cdot,s_2,\overline{\chi}), \overline{E_{\mathfrak{a}}(\cdot,s_3,\chi)E_{\mathfrak{a}}(\cdot, s_4,\overline{\chi})} \rangle_{\mathrm{reg}}$, and find a path for $(s_1,s_2,s_3,s_4)\in \mathbb{C}^4$, such that it arrives at $(\frac{1}{2}+iT,\frac{1}{2}+iT,\frac{1}{2}+iT,\frac{1}{2}+iT)$ without touching any point of singularity. 
As is discussed in Section 4 of \cite{PY}, if we further assume $w_1+w_2, w_3+w_4 \neq 1$, then we can regularize $E_{\mathfrak{a}}(\cdot, s_1,\chi)E_{\mathfrak{a}}(\cdot, s_2,\overline{\chi})$ and $E_{\mathfrak{a}}(\cdot,\overline{s_3},\overline{\chi})E_{\mathfrak{a}}(\cdot, \overline{s_4},\chi)$ respectively. That is, there exists $\mathcal{E}_1$ and $\mathcal{E}_2$ such that
$E_{\mathfrak{a}}(\cdot, s_1,\chi)E_{\mathfrak{a}}(\cdot, s_2,\overline{\chi}) - \mathcal{E}_1$ and $E_{\mathfrak{a}}(\cdot,\overline{s_3},\overline{\chi})E_{\mathfrak{a}}(\cdot, \overline{s_4},\chi)-\mathcal{E}_2 \in L^2$. Since their product is in $L^1$, we have
\begin{align*}
\langle E_{\mathfrak{a}}(\cdot, s_1,\chi)E_{\mathfrak{a}}(\cdot, s_2,\overline{\chi}), E_{\mathfrak{a}}(\cdot,\overline{s_3},\overline{\chi})E_{\mathfrak{a}}(\cdot, \overline{s_4},\chi) \rangle_{\mathrm{reg}} = I_1 + I_2,
\end{align*}
where
\begin{align*}
I_1 = I_1(s_1,s_2,s_3,s_4)= \langle E_{\mathfrak{a}}(\cdot, s_1,\chi)E_{\mathfrak{a}}(\cdot, s_2,\overline{\chi})-\mathcal{E}_1, E_{\mathfrak{a}}(\cdot,\overline{s_3},\overline{\chi})E_{\mathfrak{a}}(\cdot, \overline{s_4},\chi)-\mathcal{E}_2 \rangle,
\end{align*}
and
\begin{align*}
I_2 = I_2(s_1,s_2,s_3,s_4) =
\langle E_{\mathfrak{a}}(\cdot, s_1,\chi)E_{\mathfrak{a}}(\cdot, s_2,\overline{\chi}), \mathcal{E}_2 \rangle_{\mathrm{reg}}
+
\langle\mathcal{E}_1, E_{\mathfrak{a}}(\cdot,\overline{s_3},\overline{\chi})E_{\mathfrak{a}}(\cdot, \overline{s_4},\chi) \rangle_{\mathrm{reg}}
+
\langle\mathcal{E}_1, \mathcal{E}_2 \rangle_{\mathrm{reg}}.
\end{align*}
In order to decide $\mathcal{E}_1$ and $\mathcal{E}_2$, one needs to study carefully the behavior of $E_{\mathfrak{a}}|_{\sigma_{\mathfrak{b}}}$ for every $\mathfrak{b}$, no matter singular or not for $\chi$. According to \cite[prop 8.1]{PY}, we have
\begin{align*}
\mathcal{E}_1=  \mathcal{E}_1(s_1,s_2) = E_{\mathfrak{a}}(z,s_1+s_2) + \varphi_{\mathfrak{aa}^*}(s_1,\chi) \varphi_{\mathfrak{aa}^*}(s_2,\overline{\chi})E_{\mathfrak{a}^*}(z,2-s_1-s_2),
\end{align*}
and
\begin{align*}
\mathcal{E}_2=  \mathcal{E}_2(s_3,s_4) = E_{\mathfrak{a}}(z,\overline{s_3}+\overline{s_4}) + \varphi_{\mathfrak{aa}^*}(\overline{s_3},\overline{\chi}) \varphi_{\mathfrak{aa}^*}(\overline{s_4},\chi) E_{\mathfrak{a}^*}(z,2-\overline{s_3}-\overline{s_4}).
\end{align*}
Corollary 4.4 of [PaY] says that $\langle\mathcal{E}_1, \mathcal{E}_2 \rangle_{\mathrm{reg}}=0$, so it suffices to compute the first two terms of $I_2$.

One nice feature of the regularized integral is it is easily computable when an Eisenstein series attached to a cusp is a factor of the integrand.

Now we need to introduce the regularized integrals.
Roughly speaking, if an $SL_2(\mathbb{Z})$-automorphic function $F$ is of moderate growth, then there always exists $\mathcal{E}$, a linear combination of Eisenstein series, such that $F-\mathcal{E} = O(\sqrt{y})$ as $y\rightarrow \infty$, and still maintains automorphy. Then the convergent integral $\int (F-\mathcal{E})$ is defined to be the renormalized integral of $\int F$.
\begin{theorem}\label{RN}
For automorphic function $F$ of moderate growth and $w\neq 0,1$, we have
\begin{align*}
    \langle F, E_{\mathfrak{a}}(\cdot,w) \rangle_{\mathrm{reg}} = \int_0^{\infty} y^{w-2} (F(\sigma_{\mathfrak{a}}z) - \psi_{\mathfrak{a}}(y)) dy,
\end{align*}
where $\sigma_{\mathfrak{a}}\in SL_2(\mathbb{R})$ is any scaling matrix, and $\psi_{\mathfrak{a}}$ is the moderate growth part of $F(\sigma_{\mathfrak{a}}z)$.
\end{theorem}
\begin{corollary}\label{RNzero}
The regularized integral of any double product of Eisenstein series is zero.
\end{corollary}
\begin{remark}
All meromorphic functions in this paper is continuable. So, throughout, we directly assume a function $f(s)$ is defined on $\mathbb{C}$ from the beginning, as long as it is meromorphic on some half plane $\Re s> C$.
\end{remark}

\begin{remark}
All implied constants are assumed to be related with $\varepsilon$ and $T$ if not specified otherwise.
\end{remark}

Applying Plancherel, we have
\begin{multline*}
I_1 = \langle |E|^2-\mathcal{E}_1, |E|^2-\mathcal{E}_2 \rangle = \sum_j \sum_{u_j} \langle |E|^2-\mathcal{E}_1, u_j \rangle \langle u_j, |E|^2-\mathcal{E}_2 \rangle \\
+ \frac{1}{4\pi} \sum_{\mathfrak{b}} \int_{-\infty}^{\infty} \langle |E|^2-\mathcal{E}_1, E_{\mathfrak{b}}(\cdot, \frac{1}{2}+it) \rangle_{\mathrm{reg}} \langle E_{\mathfrak{b}}(\cdot, \frac{1}{2}+it), |E|^2-\mathcal{E}_2 \rangle_{\mathrm{reg}} dt.
\end{multline*}
Then by orthogonality and Corollary \ref{RNzero}, we can see
\begin{align*}
I_1 = \sum_j \sum_{u_j} |\langle |E|^2, u_j\rangle |^2+ \frac{1}{4\pi} \sum_{\mathfrak{b}} \int_{-\infty}^{\infty} \Big|\langle |E|^2, E_{\mathfrak{b}}(\cdot, \frac{1}{2}+it) \rangle_{\mathrm{reg}} \Big|^2 dt.
\end{align*}
Since the $L^2$-normalized $u_j$ has Fourier expansion
\begin{align*}
\rho_j \sqrt{y} \sum_{n\neq 0} \lambda_j(n) e(nx) K_{it_j}(2\pi |n| y),
\end{align*}
with $\rho_j^2 =  \frac{\cosh(\pi t_j)}{2\nu(N) L(1, \mathrm{sym}^2 u_j)}$, we arrive at \eqref{I_1}\footnote{see page 24 of \cite{PY} for a similar argument of the bad Euler factors $B(u_j,T)$, where they wrote by $F_j(A,B)$.} after unfolding and computing the period integrals (Theorem \ref{RN} for regularized inner products).

\section{Eisenstein series}
This section introduces necessary facts on Eisenstein series. Please turn to \cite{Y} and \cite{PY} for more.

For fixed $N$, we have congruence subgroup $\Gamma=\Gamma_0(N)$ of the full modular group $SL_2(\mathbb{Z})$
\begin{align*}
    \Big\{ \big(\begin{smallmatrix} a&b\\c&d \end{smallmatrix}\big) \Big| \medspace ad-bc=1, c\equiv 0 \text{ mod }N \Big\}.
\end{align*}
Let $\chi$ be a Dirichlet character of modulus $N$, suppose it is even ($\chi(-1)=1$), and define $\chi(\gamma)=\chi(d)$ where $d$ is the lower-right entry of $\gamma$. If a function $f$ is $\Gamma$-invariant, i.e., $f|_{\gamma}=f(\gamma z)=\chi(\gamma)f(z)$ for all $z\in \mathbb{H}$ and all $\gamma\in \Gamma$, then $f$ is an automorphic function of level $N$, weight $0$, and nebentypus $\chi$. 

Call an element $\mathfrak{a}$ of $\mathbb{P}^1(\mathbb{Q})$ a \textit{cusp}, and define $q_2:=\frac{N}{q_1}$ for every $q_1 \mid N$. Cusps are classified into $\Gamma_0(N)$-equivalence classes, and a nice choice of representative set is $\mathcal{C}:=\{ u/q_1 : q_1\mid N, (u,N)=1, u \text{ mod }(q_1,q_2)\}$. Write the stabilizer subgroup of $\mathfrak{a}$ in $\Gamma$ as $\Gamma_{\mathfrak{a}}$, and pick $\sigma_{\mathfrak{a}}\in SL_2(\mathbb{R})$ such that $\sigma_{\mathfrak{a}}\infty =\mathfrak{a}$ and $\sigma_{\mathfrak{a}}^{-1}\Gamma_{\mathfrak{a}}\sigma_{\mathfrak{a}}=\Gamma_{\infty}$.

The Eisenstein series of level $N$ and nebentypus $\chi \pmod{N}$ \textit{attached to the cusp} $\mathfrak{a}$ is
\begin{align*}
E_{\mathfrak{a}}(z,s, \chi)= \sum_{\gamma\in\Gamma_{\mathfrak{a}}\backslash\Gamma}\overline{\chi}(\gamma)(\text{Im}\thinspace \sigma_{\mathfrak{a}}^{-1}\gamma z)^s.
\end{align*}
The expression does not depend on the choice of $\sigma_{\mathfrak{a}}$, so it is well-defined.
As one can see from the above expression, to make this well-defined, $\mathfrak{a}$ must be \textit{singular} for $\chi$. That is, $\Gamma_{\mathfrak{a}}$ is contained in the kernel of $\chi$. It is obvious that $\Gamma$-equivalent cusps are either both, or neither singular. If we denote the subset of $\mathcal{C}$ of cusps singular for $\chi$ by $\mathcal{C}_{\chi}$, then $\mathcal{C}_{\chi}=\mathcal{C}$ if $\chi=\chi_{_0,_N}$ is the principal nebentypus, and $\mathcal{C}_{\chi}=\{1/q_1: (q_1,q_2)=1 \}$ if $\chi$ is primitive mod $N$.

Given primitive characters $\chi_j$ mod $q_j$ for $j=1,2$, we define \textit{the completed Eisenstein series attached to characters} $\chi_1,\chi_2$ by
\begin{align*}
    E_{\chi_1, \chi_2}^*(z,s)= \frac{q_2^s \pi^{-s}}{2\tau(\chi_2)} \Gamma(s)L(2s,\chi_1\chi_2) \sum_{(c,d)=1}\frac{(q_2 y)^s \chi_1(c)\chi_2(d)}{|c q_2 z +d|^{2s}}.
\end{align*}
Note that $E_{\chi_1,\chi_2}^*$ is an automorphic function of weight $0$, level $N$, and nebentypus $\chi_1\overline{\chi_2}$. 
$E^*_{\chi_1,\chi_2}(z,s)$ has three good properties.

To begin with, $E^*_{\chi_1,\chi_2}(z,s)$ has nice Fourier expansion
\begin{align*}
    e^*(y,s) + 2\sqrt{y}\sum_{n\neq 0}\lambda_{\chi_1,\chi_2}(n,s)e(nx)K_{s-\frac{1}{2}}(2\pi |n|y),
\end{align*}
where $e^*(y,s)$ is some linear combination of $y^s$ and $y^{1-s}$, $K$ is the Bessel function, and
\begin{align*}
    \lambda_{\chi_1,\chi_2}(n,s) = \chi_2(\frac{|n|}{n})\sum_{ab=|n|}\chi_1(a)\overline{\chi_2}(b)(\frac{b}{a})^{s-\frac{1}{2}}.
\end{align*}

Next, we have the following functional equation due to Huxley [H]:
\begin{align*}
    E^*_{\chi_1,\chi_2}(z,s) = E^*_{\overline{\chi_2},\overline{\chi_1}}(z,1-s).
\end{align*}

In addition, we also have in (9.1) of [Y] that
\begin{align}\label{(9.1)}
    E_{\chi_1,\chi_2}^*(\sigma_{1/q_2}z,s) &= \frac{\tau(\chi_1\chi_2)}{\tau(\chi_2)}(\frac{q_2}{N})^{s}E_{1,\chi_1\chi_2}^*(z,s);\\
    \label{(9.2)}
    E_{\overline{\chi_2},\overline{\chi_1}}^*(\sigma_{1/q_1}z,s) &= \frac{\tau(\overline{\chi_1\chi_2})}{\tau(\overline{\chi_1})}(\frac{q_1}{N})^s E_{1,\overline{\chi_1\chi_2}}^*(z,s).
\end{align}

If $\chi$ is primitive mod $N$, then $(q_1,q_2)=1$, and hence $\chi$ can uniquely be decomposed as $\chi_1\overline{\chi_2}$ for $\chi_j$ primitive mod $q_j$, $j=1,2$. By \cite[(6.2)]{Y}, we have
\begin{align}
    E_{\frac{1}{q_2}}(z,s,\chi) = N^{-s} \frac{\chi_1(-1) \tau(\chi_2)}{\Lambda(2s,\chi_1\chi_2)} E_{\chi_1,\chi_2}^*(z,s).
\end{align}

When $\chi$ is primitive mod $N$, $\mathfrak{a}=\frac{1}{q_2}\in \mathcal{C}_{\chi}$, we denote $\frac{1}{q_1}$ by $\mathfrak{a}^*$. By Corollary 3.17 and Proposition 3.18 of \cite{PY}, we have the following cuspidal behavior formulas for $\mathfrak{b}\in \mathcal{C}$ and $y\rightarrow\infty$:
\begin{align}
    E_{\mathfrak{a}}(\sigma_{\mathfrak{b}}z,s,\chi) = \begin{cases} y^s+ O(p^{-\mathbb{N}}) & \text{ if }\mathfrak{b}=\mathfrak{a}\\ \overline{\tau}(\overline{\chi_1})\tau(\chi_2) N^{-s} \frac{\Lambda(2-2s,\overline{\chi_1 \chi_2})}{\Lambda(2s,\chi_1 \chi_2 )} y^{1-s} + O(p^{-\mathbb{N}}) & \text{ if }\mathfrak{b}=\mathfrak{a}^* \\ O(p^{-\mathbb{N}}) & \text{ otherwise} \end{cases}.
\end{align}

\section{Two consequences of GRH}
This section is to justify Remark \ref{GRH}.
\begin{lemma}\label{GRH}
Assume GRH is true, $\Re s=1$ and $\chi$ is primitive mod $N$, then we have
\begin{align}
\label{GRH1}\frac{L'}{L}(s,\chi) &\ll\log\log N;\\
\label{GRH2}\frac{L''}{L}(s,\chi) &\ll (\log\log N)^{3+\varepsilon}. 
\end{align}
\end{lemma}
\begin{proof}
Inequality \eqref{GRH1} is in Theorem 5.17 of [IK]. As for \eqref{GRH2}, we firstly recall that
\begin{align*}
(\frac{L'}{L})'(s,\chi) = \frac{L''}{L}(s,\chi) - (\frac{L'}{L})^2(s,\chi).
\end{align*}
By \eqref{GRH1}, we know $(\frac{L'}{L})^2(s,\chi)\ll (\log\log N)^2$; Proposition 5.16 of [IK] says that
\begin{align*}
-\frac{L'}{L}(s,\chi) = \sum_{p} \frac{\chi(p)\log p}{p^s} \phi(\frac{p}{X}) + \sum_{\rho}\hat{\phi}(\rho-s)X^{\rho-s}+ O(1),
\end{align*}
for $X=\log^{2+\varepsilon} N(|s|+1)$, $\phi(y) = \max\{1-y,0\}$, and $\hat{\phi}(w)=w^{-1}(w+1)^{-1}$. Taking derivative on both sides, we see
\begin{align*}
\Big| (\frac{L'}{L})'(s,\chi) \Big| &\leq \sum_{p\leq X}\frac{\log^2 p}{p} + \log X \sum_{\rho}\frac{1}{|\rho-s| \cdot |\rho+1-s|} + O(1)\\
&\ll \int_2^X \frac{\log^2 t}{t} dt + \log X \cdot O(\sum_{\rho} |\rho|^{-2}).
\end{align*}
Since $\sum_{\rho} |\rho|^{-2}$ converges, above can further be bounded by $(\log\log N)^{3+\varepsilon}$. So is $\frac{L''}{L}(s,\chi)$.
\end{proof}

\section{Proof of Theorem 1}
Since we have separated the regularizec fourth moment of $E$ and proved \eqref{I_1} in Section \ref{reduction}, we only need to prove \eqref{I_2} in this section.
\subsection{Preparation}
\begin{lemma}\label{123}
Assume $w_1\neq w_2, 1-w_2$, and $w_3\neq 0,1$. Then
\begin{multline*}
    \langle E_{\chi_1,\chi_2}^*(\cdot,w_1) E_{\overline{\chi_1},\overline{\chi_2}}^*(\cdot,w_2), E_{\mathfrak{a}}(\cdot,\overline{w_3}) \rangle_{\mathrm{reg}} = N^{-w_3}q_1 (\frac{q_2}{N})^{w_1+w_2} \frac{1}{\xi(2w_3)} \prod_{p\mid N}\frac{1-p^{w_1+w_2-w_3-1}}{1-p^{-2w_3}} \\
    \xi(w_1+w_2+w_3-1)\Lambda(w_1-w_2+w_3,\chi_1\chi_2)\Lambda(-w_1+w_2+w_3,\overline{\chi_1\chi_2})\xi(-w_1-w_2+w_3+1).
\end{multline*}
\end{lemma}
\begin{proof}
Theorem \ref{RN} says that the left hand side equals
\begin{align*}
   \int_0^{\infty}y^{w_3-2} (\int_0^1 E_{\chi_1,\chi_2}^*(\sigma_{\mathfrak{a}}z,w_1) E_{\chi_1,\chi_2}^*(\sigma_{\mathfrak{a}}z,w_2)) dx - \Psi) dy,
\end{align*}
where $\Psi$ is the moderate growth part of $F=E_{\chi_1,\chi_2}^*(z,w_1)E_{\chi_1,\chi_2}^*(\cdot,w_2)$. Applying (\ref{(9.1)}), we can rewrite it with
\begin{align*}
    q_1 (\frac{q_2}{N})^{w_1+w_2}\int_0^{\infty}y^{w_3-2} (\int_0^1 E_{1,\chi_1\chi_2}^*(z,w_1) E_{\chi_1,\chi_2}^*(\cdot,w_2) dx - \Psi) dy.
\end{align*}

Since the moderate growth part of $F$ is exactly the product of the moderate growth parts of $E_{1,\chi_1\chi_2}^*(z,w_1)$ and $E_{1,\overline{\chi_1\chi_2}}^*(z,w_2)$, we further write it as (note $\lambda_{1,\chi_1\chi_2}(-1)+\lambda_{1,\overline{\chi_1}\overline{\chi_2}}(-1)=2$ by evenness of $\chi_1\chi_2$)
\begin{multline*}
    4q_1 (\frac{q_2}{N})^{w_1+w_2} \int_0^{\infty}y^{w_3-1} \sum_{n\neq 0}\lambda_{1,\chi_1\chi_2}(n,w_1)\lambda_{1,\overline{\chi_1\chi_2}}(-n,w_2) K_{w_1-\frac{1}{2}}(2\pi|n|y) K_{w_2-\frac{1}{2}}(2\pi|n|y) dy\\
    =8q_1 (\frac{q_2}{N})^{w_1+w_2} \sum_{n\geq 1}\frac{\lambda_{1,\chi_1\chi_2}(n,w_1)\lambda_{1,\overline{\chi_1\chi_2}}(n,w_2)}{n^{w_3}} \int_0^{\infty} y^{w_3-1} K_{w_1-\frac{1}{2}}(2\pi y) K_{w_2-\frac{1}{2}}(2\pi y) dy.
\end{multline*}
By (6.576.4) of [GR], we know the integral equals
\begin{align*}
    \frac{\pi^{-w_3}}{8 \Gamma(w_3)} \prod_{\epsilon_1=\pm 1}\prod_{\epsilon_2=\pm 1} \Gamma(\frac{w_3 +\epsilon_1(w_1-\frac{1}{2})+\epsilon_2(w_2-\frac{1}{2})}{2}).
\end{align*}
while the Dirichlet factors into (see (13.1) of [I1])
\begin{align*}
    \frac{\zeta(w_1+w_2+w_3-1) L(w_1-w_2+w_3,\chi_1\chi_2) L(-w_1+w_2+w_3,\overline{\chi_1\chi_2}) L(-w_1-w_2+w_3+1,\chi_{_0,_N})}{L(2w_3,\chi_{_0,_N})}.
\end{align*}
Completing these $L$-functions with proper factors, we obtain the right hand side.
\end{proof}
\begin{proposition}
We have
\begin{align*}
    I_2
    = &H_1 \xi(s_1+s_2+s_3+s_4-1) \xi(-s_1-s_2+s_3+s_4+1) \\
    + &H_2 \xi(-s_1-s_2-s_3-s_4+3) \xi(s_1+s_2-s_3-s_4+1) \\
    + &H_3 \xi(s_1+s_2+s_3+s_4-1) \xi(s_1+s_2-s_3-s_4+1) \\
    + &H_4 \xi(-s_1-s_2-s_3-s_4+3) \xi(-s_1-s_2+s_3+s_4+1),
\end{align*}
where $H_j=H_j(s_1,s_2,s_3,s_4)$ for $j=1,2,3,4$ with
\begin{multline*}
    H_1 = N^{-s_1-s_2-s_3-s_4+1} \frac{\Lambda(s_1-s_2+s_3+s_4,\chi_1\chi_2)\Lambda(-s_1+s_2+s_3+s_4, \overline{\chi_1\chi_2})}{\xi(2s_3+2s_4) \Lambda(2s_1,\chi_1\chi_2)\Lambda(2s_2,\overline{\chi_1\chi_2})} \\
    \prod_{p\mid N}\frac{1-p^{s_1+s_2-s_3-s_4-1}}{1-p^{-2s_3-2s_4}},
\end{multline*}
\begin{multline*}
H_2 = N^{-1} \frac{\Lambda(s_1-s_2-s_3-s_4+2,\chi_1\chi_2)\Lambda(-s_1+s_2-s_3-s_4+2,\overline{\chi_1\chi_2})}{\xi(-2s_3-2s_4+4) \Lambda(2s_1,\chi_1\chi_2)\Lambda(2s_2,\overline{\chi_1\chi_2})} \\
\frac{\Lambda(2-2s_3,\overline{\chi_1\chi_2}) \Lambda(2-2s_4,\chi_1\chi_2)}{\Lambda(2s_3,\chi_1\chi_2)\Lambda(2s_4,\overline{\chi_1\chi_2})} \prod_{p\mid N}\frac{1-p^{-s_1-s_2+s_3+s_4-1}}{1-p^{2s_3+2s_4-4}},
\end{multline*}
\begin{multline*}
    H_3 = N^{-s_1-s_2-s_3-s_4+1} \frac{\Lambda(s_1+s_2+s_3-s_4,\chi_1\chi_2)\Lambda(s_1+s_2-s_3+s_4, \overline{\chi_1\chi_2})}{\xi(2s_1+2s_2) \Lambda(2s_3,\chi_1\chi_2) \Lambda(2s_4,\overline{\chi_1\chi_2})} \\
    \prod_{p\mid N}\frac{1-p^{-s_1-s_2+s_3+s_4-1}}{1-p^{-2s_1-2s_2}},
\end{multline*}
\begin{multline*}
H_4 = N^{-1} \frac{\Lambda(-s_1-s_2+s_3-s_4+2,\chi_1\chi_2)\Lambda(-s_1-s_2-s_3+s_4+2,\overline{\chi_1\chi_2})}{\xi(-2s_1-2s_2+4) \Lambda(2s_3,\chi_1\chi_2)\Lambda(2s_4,\overline{\chi_1\chi_2})} \\
\frac{\Lambda(2-2s_1,\overline{\chi_1\chi_2})\Lambda(2-2s_2,\chi_1\chi_2)}{\Lambda(2s_1,\chi_1\chi_2)\Lambda(2s_2,\overline{\chi_1\chi_2})} \prod_{p\mid N}\frac{1-p^{s_1+s_2-s_3-s_4-1}}{1-p^{2s_1+2s_2-4}}.
\end{multline*}
\end{proposition}
\begin{proof}
Since $E_{\mathfrak{a}}(z,s,\chi)=N^{-s}\chi_1(-1)\rho_{\chi_1,\chi_2}(s)E^*_{\chi_1,\chi_2}(z,s)$ with
\begin{align*}
\rho_{\chi_1,\chi_2}(s):=\frac{q_1^s \tau(\chi_2)}{\Lambda(2s,\chi_1\chi_2)},
\end{align*}
we see that $ \langle E_{\mathfrak{a}}(\cdot,s_1,\chi) E_{\mathfrak{a}}(\cdot, s_2,\overline{\chi}), E_{\mathfrak{a}}(\cdot, s_3+ s_4) \rangle_{\mathrm{reg}}$ equals
\begin{align*}
N^{-s_1-s_2}\rho_{\chi_1,\chi_2}(s_1)\rho_{\overline{\chi_1},\overline{\chi_2}}(s_2)\langle E^*_{\chi_1,\chi_2}(\cdot,s_1) \overline{E^*_{\chi_1,\chi_2}(\cdot,s_2)}, E_{\mathfrak{a}}(\cdot, s_3+ s_4) \rangle_{\mathrm{reg}}.
\end{align*}
Applying Lemma \ref{123} with $w_1=s_1, w_2=s_2$ and $w_3=s_3+s_4$, we see above further equals
\begin{multline*}
\frac{N^{-s_1-s_2-s_3-s_4}q_1^{s_1+s_2-s_3-s_4+1}\tau(\chi_2)\overline{\tau(\chi_2)}}{\xi(2s_3+2s_4) \Lambda(2s_1,\chi_1\chi_2)\Lambda(2s_2,\overline{\chi_1\chi_2})} \Lambda(s_1-s_2+s_3+s_4,\chi_1\chi_2)\Lambda(-s_1+s_2+s_3+s_4, \overline{\chi_1\chi_2}) \\
\xi(s_1+s_2+s_3+s_4-1)\xi(-s_1-s_2+s_3+s_4+1) \prod_{p\mid N}\frac{1-p^{s_1+s_2-s_3-s_4-1}}{1-p^{-2s_3-2s_4}}.
\end{multline*}
This accounts for the first term $H_1 \xi(s_1+s_2+s_3+s_4-1) \xi(-s_1-s_2+s_3+s_4+1)$. The other three terms can be obtained in the same way, except that we adopt \eqref{(9.2)} in place of $\eqref{(9.1)}$ for the second and fourth terms.
\end{proof}
\subsection{Estimation of $I_2$}
Now set $s_1=s_3=\frac{1}{2}+iT$, $s_2=\frac{1}{2}+\eta'-iT$ and $s_4=\frac{1}{2}+\eta-iT$ with $0<\eta'<\eta<\frac{1}{4}$. Under limit $\eta'\rightarrow 0$, $I_2$ tends to
\begin{align*}
\underbrace{F_1(\eta) \xi^2(1+\eta)}_{\Xi_1} + \underbrace{F_2(\eta) \xi^2(1-\eta)}_{\Xi_2} + \underbrace{F_3(\eta) \xi(1+\eta)\xi(1-\eta)}_{\Xi_3} + \underbrace{F_4(\eta) \xi(1-\eta)\xi(1+\eta)}_{\Xi_4},
\end{align*}
with $F_j(\eta)=\lim_{\eta'\rightarrow 0^+}H_j(\frac{1}{2}+iT,\frac{1}{2}+\eta'-iT,\frac{1}{2}+iT,\frac{1}{2}+\eta-iT)$. The explicit forms are
\begin{align*}
F_1 &= N^{-1-\eta}\frac{|\Lambda(1+\eta+2iT,\chi_1\chi_2)|^2}{\xi(2+2\eta) |\Lambda(1+2iT,\chi_1\chi_2)|^2} \prod_{p\mid N} (1+\frac{1}{p^{1+\eta}})^{-1}; \\
F_2 &= N^{-1} \frac{|\Lambda(1-\eta+2iT,\chi_1\chi_2)|^2 \Lambda(1-2\eta+2iT,\chi_1\chi_2)}{\xi(2-2\eta)\Lambda^2(1+2iT,\overline{\chi_1\chi_2}) \Lambda(1+2\eta-2iT,\overline{\chi_1\chi_2})} \prod_{p\mid N} (1+\frac{1}{p^{1-\eta}})^{-1}; \\
F_3 &= N^{-1-\eta} \frac{\Lambda(1-\eta+2iT,\chi_1\chi_2)\Lambda(1+\eta-2iT,\overline{\chi_1\chi_2})}{\xi(2) \Lambda(1+2iT,\chi_1\chi_2)\Lambda(1+2\eta-2iT,\overline{\chi_1\chi_2})} \prod_{p\mid N}\frac
{1-p^{-1+\eta}}{1-p^{-2}}; \\
F_4 &= N^{-1} \frac{\Lambda(1-\eta+2iT,\chi_1\chi_2)\Lambda(1+\eta-2iT,\overline{\chi_1\chi_2})}{\xi(2) \Lambda(1+2iT,\chi_1\chi_2)\Lambda(1+2\eta-2iT,\overline{\chi_1\chi_2})} \prod_{p\mid N}\frac{1-p^{-1-\eta}}{1-p^{-2}}.
\end{align*}

Further calculation shows $F_1(0)=F_2(0)=F_3(0)=F_4(0)=(\xi(2)\nu(N))^{-1}$, and
\begin{align*}
F_1'(0) &= F_1(0) \cdot \Big(-\log N &&+ 2\Re \frac{\Lambda'}{\Lambda}(1+2iT,\chi_1\chi_2) &&-2\frac{\xi'}{\xi}(2) &&+ \sum_{p\mid N}\frac{\log p}{p+1} \Big); \\
F_2'(0) &= F_2(0) \cdot \Big(&&- 6\Re\frac{\Lambda'}{\Lambda}(1+2iT,\chi_1\chi_2) &&+2\frac{\xi'}{\xi}(2) &&- \sum_{p\mid N}\frac{\log p}{p+1} \Big); \\
F_3'(0) &= F_3(0) \cdot \Big(-\log N &&-2\Re \frac{\Lambda'}{\Lambda}(1+2iT,\chi_1\chi_2) && &&- \sum_{p\mid N}\frac{\log p}{p+1} \Big); \\
F_4'(0) &= F_4(0) \cdot \Big(&&-2\Re\frac{\Lambda'}{\Lambda}(1+2iT,\chi_1\chi_2) && &&+\sum_{p\mid N}\frac{\log p}{p+1} \Big).
\end{align*}

Moreover, by $F_j''=(F_j\cdot \frac{F_j'}{F_j})' = F_j \cdot ((\frac{F_j'}{F_j})^2 + (\frac{F_j'}{F_j})')$, we see that
\begin{align*}
F_1''(0) 	&= F_1(0) \cdot \Big(-\log N + 2\Re \frac{\Lambda'}{\Lambda}(1+2iT,\chi_1\chi_2) -2\frac{\xi'}{\xi}(2) + \sum_{p\mid N}\frac{\log p}{p+1} \Big)^2 \\
		&+ F_1(0) \cdot \Big(\log^2 N + 2\Re (\frac{\Lambda'}{\Lambda})'(1+2iT,\chi_1\chi_2) -4(\frac{\xi'}{\xi})'(2) - \sum_{p\mid N}\frac{p \log^2 p}{(p+1)^2} \Big); \\
F_2''(0)	&= F_2(0) \cdot \Big(- 6\Re\frac{\Lambda'}{\Lambda}(1+2iT,\chi_1\chi_2) +2\frac{\xi'}{\xi}(2) - \sum_{p\mid N}\frac{\log p}{p+1} \Big)^2 \\
		&+ F_2(0) \cdot \Big(2(\frac{\Lambda'}{\Lambda})'(1+2iT,\chi_1\chi_2) + 6\Im(\frac{\Lambda'}{\Lambda})'(1+2iT,\chi_1\chi_2) -4(\frac{\xi'}{\xi})'(2) - \sum_{p\mid N}\frac{p \log^2 p}{(p+1)^2} \Big); \\
F_3''(0)	&= F_3(0) \cdot \Big(-\log N -2\Re \frac{\Lambda'}{\Lambda}(1+2iT,\chi_1\chi_2) - \sum_{p\mid N}\frac{\log p}{p-1} \Big)^2 \\
		&+ F_3(0) \cdot \Big(\log^2 N +2\Re (\frac{\Lambda'}{\Lambda})'(1+2iT,\chi_1\chi_2) -4(\frac{\Lambda'}{\Lambda})'(1-2iT,\overline{\chi_1\chi_2}) + \sum_{p\mid N}\frac{p\log^2 p}{(p-1)^2} \Big); \\
F_4''(0)	&= F_4(0) \cdot  \Big(-2\Re\frac{\Lambda'}{\Lambda}(1+2iT,\chi_1\chi_2) +\sum_{p\mid N}\frac{\log p}{p-1} \Big)^2 \\
		&+F_4(0) \cdot \Big(2\Re (\frac{\Lambda'}{\Lambda})'(1+2iT,\chi_1\chi_2) -4(\frac{\Lambda'}{\Lambda})'(1-2iT,\overline{\chi_1\chi_2})+ \sum_{p\mid N}\frac{p\log^2 p}{(p-1)^2} \Big).
\end{align*}

Let $\xi(s)=(s-1)^{-1} + a +b(s-1) + O((s-1)^2)$ for some $a,b$, then for $\eta$ around $0$, we have
\begin{align*}
&\Xi_1 &&= +\frac{F_1(0)}{\eta^2} &&+ \frac{F_1'(0)+2aF_1(0)}{\eta} &&+ \frac{F''_1(0)}{2}+2aF'_1(0) &&+ (a^2+2b)F_1(0) &&+ O(\eta); \\
&\Xi_2 &&= +\frac{F_2(0)}{\eta^2} &&+ \frac{F_2'(0)-2aF_2(0)}{\eta} &&+ \frac{F''_2(0)}{2}-2aF'_2(0) &&+ (a^2+2b)F_2(0) &&+ O(\eta); \\
&\Xi_3 &&= -\frac{F_3(0)}{\eta^2} &&- \frac{F_3'(0)}{\eta} &&- \frac{F_3''(0)}{2} &&+ (a^2-2b)F_3(0) &&+ O(\eta); \\
&\Xi_4 &&= -\frac{F_4(0)}{\eta^2} &&- \frac{F_4'(0)}{\eta} &&- \frac{F_4''(0)}{2} &&+ (a^2-2b)F_4(0) &&+ O(\eta).
\end{align*}

Thus, the coefficients of the $\frac{1}{\eta^2}$ and $\frac{1}{\eta}$ of $I_2$ vanish by cancellation, and its constant term equals
\begin{multline*}
\frac{1}{2}(F_1''(0)+F_2''(0)-F_3''(0)-F_4''(0)) + 2a(F_1'(0)-F_2'(0)) + 4a^2F_1(0) \\
= \frac{1}{2}(F_1''(0)+F_2''(0)-F_3''(0)-F_4''(0)) + \frac{4a^2 -2a \log N +16a \Re \frac{\Lambda'}{\Lambda}(1+2iT,\chi_1\chi_2)}{\xi(2)\nu(N)}.
\end{multline*}
A well-known fact being $\sum_{p\mid N}\frac{\log p}{p} = O(\log\log N)$, we have
\begin{multline*}
\xi(2)\nu(N) I_2 = 4 \Re(\frac{\Lambda'}{\Lambda})'(1+2iT,\chi_1\chi_2) + 16 \Re (\frac{\Lambda'}{\Lambda})^2  (1+2iT,\chi_1\chi_2)+16\Big|\frac{\Lambda'}{\Lambda}\Big|^2(1+2iT,\chi_1\chi_2) \\
-4\log N \Re\frac{\Lambda'}{\Lambda}(1+2iT,\chi_1\chi_2) + O\Big(\log N + \log\log N \Big| \Re \frac{\Lambda'}{\Lambda}(1+2iT,\chi_1\chi_2) \Big|\Big).
\end{multline*}
Since $(\frac{\Lambda'}{\Lambda})'=\frac{\Lambda''}{\Lambda} - (\frac{\Lambda'}{\Lambda})^2$, and
\begin{multline*}PY
\frac{\Lambda''}{\Lambda}(s,\chi_1\chi_2) = \frac{1}{4}\log^2 \frac{N}{\pi} + \frac{1}{4}\frac{\Gamma''}{\Gamma}(\frac{s}{2})+\frac{L''}{L}(s,\chi_1\chi_2) \\
+ \frac{1}{2}\log\frac{N}{\pi}\frac{\Gamma'}{\Gamma}(\frac{s}{2}) + \log \frac{N}{\pi}\frac{L'}{L}(s,\chi_1\chi_2)+\frac{\Gamma'}{\Gamma}(\frac{s}{2})\frac{L'}{L}(s,\chi_1\chi_2),
\end{multline*}
we can see that
\begin{multline*}
    \xi(2)\nu(N) I_2 = \log^2 N + 4\Re \frac{L''}{L}(1+2iT,\chi_1\chi_2) + 12 \Re (\frac{\Lambda'}{\Lambda})^2  (1+2iT,\chi_1\chi_2)+16\Big|\frac{\Lambda'}{\Lambda}\Big|^2(1+2iT,\chi_1\chi_2) \\
-4\log N \Re\frac{\Lambda'}{\Lambda}(1+2iT,\chi_1\chi_2) + O\Big(\log N \Big| \frac{L'}{L}(s,\chi_1\chi_2) \Big| + \log\log N \Big| \Re \frac{\Lambda'}{\Lambda}(1+2iT,\chi_1\chi_2) \Big|\Big),
\end{multline*}
and after substituting $\frac{\Lambda'}{\Lambda}(s,\chi_1\chi_2) = \frac{1}{2}\log N + \frac{1}{2}\frac{\Gamma'}{\Gamma}(\frac{s}{2}) + \frac{L'}{L}(s,\chi_1\chi_2)$, we arrive at
\begin{multline*}
    \xi(2)\nu(N) I_2 = 4\log^2 N + 4\Re\frac{L''}{L}(1+2iT,\chi_1\chi_2) \\
    +O\Big(\log N \Big| \frac{L'}{L}(s,\chi_1\chi_2) \Big|+ \Big| \frac{L'}{L}(s,\chi_1\chi_2) \Big|^2 + \log\log N \Big| \Re \frac{\Lambda'}{\Lambda}(1+2iT,\chi_1\chi_2) \Big|\Big).
\end{multline*}
Noting $\xi(2)=\frac{\pi}{6}$ and $\tfrac{\Lambda'}{\Lambda} = \tfrac{L'}{L} + \log N + O_{_T}(1)$, we complete the proof.

\end{document}